\documentclass{amsart}
\usepackage{
    mathtools,
    hyperref,
    tikz-cd,
    enumitem,
}

\setlist[enumerate]{label=\textup{(\roman*)}}

\numberwithin{equation}{section}
\numberwithin{table}   {section}

\newtheorem{thm}        {Theorem}[section]
\newtheorem{lem}  [thm] {Lemma}
\newtheorem{cor}  [thm] {Corollary}
\newtheorem{prop} [thm] {Proposition}
\theoremstyle{definition}
\newtheorem{dfn}  [thm] {Definition}

\newtheorem{note} [thm] {Notes}
\newtheorem{nts}  [thm] {Notation}

\DeclareMathOperator{\aCyc}    {aCyc}
\DeclareMathOperator{\codim}   {codim}
\DeclareMathOperator{\radikal} {Rad}
\DeclareMathOperator{\tr}      {tr}

\newcommand{\nat} {\mathbb{N}}
\newcommand{\zahl}{\mathbb{Z}}
\newcommand{\blank}{{-}}
\newcommand{\id}       {\ensuremath{{1 \le i \le \ell(A)}}} % index domain for i that appears in the paper many times
\newcommand{\Cyc}   [2]{\ensuremath{\operatorname{Cyc}^{#1}(#2)}}

\newcommand{\cartan}[1]{\ensuremath{C_{#1}}}
\newcommand{\komm}  [1]{\ensuremath{K(#1)}}
\newcommand{\kr}    [2]{\ensuremath{KR^{#1}(#2)}}
\newcommand{\rad}   [2]{\ensuremath{\radikal^{#1}(#2)}}
\newcommand{\Ext}   [1]{\operatorname{Ext}^{#1}}

\begin{document}

% Front matter
\title
    [On theorems of {B}rauer-{N}esbitt and {B}randt]
    {On theorems of {B}rauer-{N}esbitt and {B}randt for\\ characterizations of small block algebras}

\author[S.~Koshitani]{Shigeo Koshitani}
\address{Center for Frontier Science, Chiba University, 1-33 Yayoi-cho, Inage-ku, Chiba, 263-8522, Japan.}
\email{koshitan@math.s.chiba-u.ac.jp}

\author[T.~Sakurai]{Taro Sakurai}
\address{Department of Mathematics and Informatics, Graduate School of Science, Chiba University, 1-33, Yayoi-cho, Inage-ku, Chiba-shi, Chiba, 263-8522 Japan}
\email{tsakurai@math.s.chiba-u.ac.jp}

\thanks{The first author was partially supported by the Japan Society for
Promotion of Science (JSPS), Grant-in-Aid for Scientific Research
(C)15K04776, 2015--2018}

 % Modular representations and characters < Representation theory of groups
%\date{\currenttime\ \today}
\date{\today}
\begin{abstract}
    \noindent % The first line of the first paragraph is not indented, but first lines of subsequent paragraphs are.
    In 1941, Brauer-Nesbitt established a characterization of a block with trivial defect group as a block $B$ with $k(B) = 1$ where $k(B)$ is the number of irreducible ordinary characters of $B$.
    In 1982, Brandt established a characterization of a block with defect group of order two as a block $B$ with $k(B) = 2$.
    These correspond to the cases when the block is Morita equivalent to the one-dimensional algebra and to the non-semisimple two-dimensional algebra, respectively.

    In this paper, we redefine $k(A)$ to be the codimension of the commutator subspace $K(A)$ of a finite-dimensional algebra $A$ and
	prove analogous statements for arbitrary (not necessarily symmetric) finite-dimensional algebras.
    This is achieved by extending the Okuyama refinement of the Brandt result to this setting.
    To this end, we study the codimension of the sum of the commutator subspace $K(A)$ and $n$th Jacobson radical $\operatorname{Rad}^n(A)$.
    We prove that this is Morita invariant and give an upper bound for the codimension as well.

\end{abstract}

\maketitle

\tableofcontents

% Body
\section{Introduction}
\label{sec: introduction}

\noindent % the very first line of the paper is not indented
Modular representation theory of finite groups aims to understand
the structure of a block of a finite group algebra and its invariants.
The complexity of the representations of a block \( B \) is measured by invariants including
the defect group \( P \),
the number of irreducible ordinary characters \( k(B) \), the number of irreducible modular characters \( \ell(B) \), and
the Cartan matrix \cartan{B} of the block \( B \).
A semisimple block, a block of the simplest kind, is characterized as a block of defect zero;
The characterization has the following equivalent condition, which was essentially due to Brauer-Nesbitt~\cite[Section~11]{BN41}.
\begin{equation}
    P \cong 1            \iff k(B) = 1.
\end{equation}
A few decades later Brandt~\cite[Theorem~A]{Bra82} proved similar.
\begin{equation}
    P \cong \zahl/2\zahl \iff \text{\( k(B) = 2 \) (and \( \ell(B) = 1 \))}.
\end{equation}
This condition characterizes a block that is Morita equivalent to the group algebra of $\zahl/2\zahl$.
For textbook accounts for these, see the Landrock textbook~\cite[Section~I.16]{Lan83}.
We call a block $B$ small if $k(B)$ is small.

The aim of this paper is to gain a better understanding of the above results
by generalizing to arbitrary finite-dimensional algebras.
Throughout this paper,
\( A \) denotes a finite-dimensional algebra over a field \( F \) and
 \( \{\, S_i \mid \id \,\} \) a complete set of pairwise non-isomorphic simple right \( A \)-modules.
An obstacle for generalization is that there is no established notion of \( k(A) \).
To overcome this difficulty we introduce the following definition.

\begin{dfn}
    Let \komm{A} be the \emph{commutator subspace} of \( A \)
    (i.e., the \( F \)-subspace of \( A \) spanned by \( xy - yx \) for all \( x, y \in A \)).
    Then we define
    \begin{equation}
        k(A) \coloneqq \codim \komm{A}.
    \end{equation}
\end{dfn}

The new definition of \( k(A) \) for an algebra \( A \) coincides with the old \( k(B) \) for a block \( B \) by K\"ulshammer~\cite[Lemma~A(ii)]{Kue81}.
Hence \( k(A) \) is supposed to measure the complexity of the representations of the algebra \( A \).
Indeed, we are able to prove that this is the case if \( k(A) \) is small.

\begin{thm}[See also Theorem~\ref{thm: Chlebowitz}]
    \label{thm: k = 1}
    Suppose that \( F \) is a splitting field for \( A \).
    Then the following statements are equivalent.
    \begin{enumerate}
        \item \( A \) is Morita equivalent to \( F \).
            \label{item: Morita equiv. to F}
        \item \( k(A) = 1 \).
            \label{item: k = 1}
    \end{enumerate}
\end{thm}

\begin{thm}[See also Theorem~\ref{thm: Chlebowitz}]
    \label{thm: k = 2 and l = 1}
    Suppose that \( F \) is a splitting field for \( A \).
    Then the following statements are equivalent.
    \begin{enumerate}
        \item \( A \) is Morita equivalent to \( F[X]/(X^2) \).
            \label{item: Morita equiv. to F[X]/(X^2)}
        \item \( k(A) = 2 \) and \( \ell(A) = 1 \).
            \label{item: k = 2 and l = 1}
    \end{enumerate}
\end{thm}

Note that we cannot omit the condition \( \ell(A) = 1 \) in the above
(Note~\ref{note: k = l = 2}).
If, among others, \( A \) is non-semisimple symmetric then we can omit the condition (Proposition~\ref{prop: k = l}).
We also remark that a natural candidate \( k^*(A) \coloneqq \dim Z(A) \), the dimension of the center, does not work here.
Consider the \( F \)-algebra \( T_n \) of lower triangular matrices of degree \( n \in \nat \).
Then \( k^*(T_n) = 1 \) and the analog is no longer true.

These theorems are obtained from a study of the codimension of the \( F \)-subspace of \( A \) defined by
\begin{equation}
    \kr{n}A \coloneqq \komm{A} + \rad{n}A \qquad (n \in \nat)
\end{equation}
where \rad{n}{A} denotes the \( n \)th Jacobson radical of \( A \).
The following has been known for finite-dimensional symmetric algebras, while we extend it for \emph{arbitrary} finite-dimensional algebras.

\begin{thm}[See also Note~\ref{note: about theorem}]
    \label{thm: Brauer Okuyama}
    Suppose that \( F \) is a splitting field for \( A \).
    Then the following holds.
    \begin{enumerate}
        \item \( \codim \kr{1}A  = \ell(A) \).
            \label{item: Brauer}
        \item \( \codim \kr{2}A  = \ell(A) + \sum_\id \dim \Ext1(S_i, S_i) \).
            \label{item: Okuyama}
        \item \( \ell(A) + \sum_\id \dim \Ext1(S_i, S_i) \le k(A) \le \tr \cartan{A} \).
            \label{item: bound for k}
    \end{enumerate}
\end{thm}

This paper is organized as follows.
First, Section~\ref{sec: Morita invariance} is devoted to prove Morita invariance (Theorem~\ref{thm: Morita invariance}).
In Section~\ref{sec: upper bound}, we establish an upper bound for the codimension of \( \kr{n}A \) in Theorem~\ref{thm: Otokita} and prove Theorem~\ref{thm: Brauer Okuyama}.
Finally, Section~\ref{sec: small algebras} provides a characterization of truncated polynomial algebras (Theorem~\ref{thm: truncated poly. ring}) using Theorems~\ref{thm: Brauer Okuyama} and \ref{thm: Morita invariance}.
Then characterizations of small algebras (Theorems~\ref{thm: k = 1} and \ref{thm: k = 2 and l = 1}) follow immediately.

\section{Morita invariance}
\label{sec: Morita invariance}
This section is devoted to prove Morita invariance of the codimension of \kr{n}{A}.
See also Sakurai~\cite[Theorem~2.3]{Sak17} for its ``dual''.

\begin{thm}
    \label{thm: Morita invariance}
    Let \( A \) and \( B \) be Morita equivalent finite-dimensional algebras over a field \( F \).
    Then there is an \( F \)-linear isomorphism
    \( A/\komm{A} \to B/\komm{B}  \)
    inducing the \( F \)-linear isomorphism
    \( A/\kr{n}A  \to B/\kr{n}B \)
    for every \( n \in \nat \).
\end{thm}

\begin{proof}
    It suffices to prove the case for \( A \) and its basic algebra \( B \coloneqq eAe \)
    because Morita equivalent basic algebras are isomorphic.
    Since the basic idempotent \( e \in A \) is full, there exist elements \( u_k, v_k \in A \) such that
    \begin{equation}
        \label{eq: 1 = ...}
        1 = \sum_k u_k e v_k.
    \end{equation}
    Then define \( F \)-linear maps
    \begin{equation*} \begin{tikzcd}
        A/\komm{A} \arrow[r, "\tau",   shift left]
        &
        B/\komm{B} \arrow[l, "\sigma", shift left]
    \end{tikzcd} \end{equation*}
    by
    \begin{align*}
        \tau  \big(a + \komm{A}\big) &= \sum_k e v_k a u_k e + \komm{B} & &(a \in A)
        \\
        \sigma\big(b + \komm{B}\big) &= b                    + \komm{A} & &(b \in B).
    \end{align*}
    We claim that \( \tau \) and \( \sigma \) are well-defined and mutually inverse.
    Indeed, for \( x, y \in A \) we have
    \begin{align*}
        \sum_k e v_k (xy - yx) u_k e
            &= \sum_k e v_k xy u_k e
                - \sum_l e v_l yx u_l e
            \\
            &= \sum_{k,l} e v_k x u_l e v_l y u_k e
                - \sum_{l,k} e v_l y u_k e v_k x u_l e
            \\
            &= \sum_{k,l} \big((e v_k x u_l e)(e v_l y u_k e)
                - (e v_l y u_k e)(e v_k x u_l e)\big)
            \\
            &\in \komm{B}.
    \end{align*}
    Hence \( \tau \) is well-defined.
	It is routine to check the other parts.

    These linear isomorphisms induce linear maps \( \tau_n \) and \( \sigma_n \) commuting the following diagram
    where \( \pi_{n, \blank} \) denotes the canonical epimorphism.
    \begin{equation*} \begin{tikzcd}
        A/\komm{A} \arrow[r, "\tau",     shift left] \arrow[d, "\pi_{n, A}", two heads]
        &
        B/\komm{B} \arrow[l, "\sigma",   shift left] \arrow[d, "\pi_{n, B}", two heads]
        \\
        A/\kr{n}A  \arrow[r, "\tau_n",   shift left, dashed]
        &
        B/\kr{n}B  \arrow[l, "\sigma_n", shift left, dashed]
    \end{tikzcd} \end{equation*}

    These maps are well-defined and
    diagram chase reveals that these maps are mutually inverse
    since \( \pi_{n, \blank} \) is an epimorphism.
\end{proof}

\section{Upper bound}
\label{sec: upper bound}
In this section we establish an upper bound for the codimension of \kr{n}{A},
which extends the result of Otokita~\cite{Oto16} (Theorem~\ref{thm: Otokita}).
As a corollary we prove Theorem~\ref{thm: Brauer Okuyama}.

\begin{nts}
    For a basic\footnote{See \cite[p.~305]{AF92}.} set of primitive idempotents
    \( \{\, e_i \mid \id \,\}\)
    of \( A \), set
    \begin{align*}
        % subspace spanned by oriented paths that not cycles (cf. acyclic)
        \aCyc(A)
        \coloneqq
        \sum_{\substack{1 \le i, j \le \ell(A)\\ i \neq j}} e_i A e_j
        \quad \text{and} \quad
        % subspace spanned by oriented cycles of length greater than or equal to n
        \Cyc{{} \ge n}{A}
        \coloneqq
        \sum_\id                                            e_i \rad{n}{A} e_i.
    \end{align*}
\end{nts}

\begin{lem}
    \label{lem: rewrite ubd}
    If \( A \) is basic then
    \begin{equation*}
        \sum_\id \dim e_i A e_i / e_i \rad{n}{A} e_i
        =
        \codim \big( \aCyc(A) + \Cyc{{} \ge n}{A} \big).
    \end{equation*}
\end{lem}

\begin{proof}
    \begin{align*}
        &\sum_i \dim e_i A e_i / e_i \rad{n}{A} e_i
        \\
        &\qquad=
        \dim \sum_i  e_i A e_i \big/ \sum_i e_i \rad{n}{A} e_i
        \\
        &\qquad=
        \dim
            \sum_{i, j}  e_i A e_j
            \bigg/
            \bigg( \sum_{i \neq j} e_i A e_j + \sum_i e_i \rad{n}{A} e_i \bigg)
        \\
        &\qquad= % basic
        \codim \big( \aCyc(A) + \Cyc{{} \ge n}{A} \big).
        \qedhere
    \end{align*}
\end{proof}

\begin{prop}
    \label{prop: equal iff}
    If \( A \) is basic then the following statements are equivalent.
    \begin{enumerate}
        \item
            \(
                \codim \kr{n}A
                =
                \sum_\id \dim e_i A e_i / e_i \rad{n}{A} e_i.
            \)
            \label{item: equality holds}

        \item
            \(
                \komm{A}
                \subseteq
                \aCyc(A) + \Cyc{{} \ge n}{A}.
            \)
            \label{item: K in aCyc + Cyc}
    \end{enumerate}
\end{prop}

\begin{proof}
    First, note that
    \begin{equation}
        \label{eq: trivial inclusion}
        \kr{n}A
        \supseteq
        \aCyc(A) + \Cyc{{} \ge n}{A}.
    \end{equation}

    \ref{item: equality holds} \( \implies \) \ref{item: K in aCyc + Cyc}:
    By \eqref{eq: trivial inclusion} and Lemma~\ref{lem: rewrite ubd}, we have \( \komm{A} \subseteq \kr{n}A = \aCyc(A) + \Cyc{{} \ge n}{A} \).

    \ref{item: K in aCyc + Cyc} \( \implies \) \ref{item: equality holds}:
    By the hypothesis, we have the following.
    \begin{align*}
        \kr{n}A
        &\subseteq
        \aCyc(A) + \Cyc{{} \ge n}{A} + \rad{n}A
        \\
        &=
        \aCyc(A) + \sum_{i \neq j} e_i \rad{n}{A} e_j + \Cyc{{} \ge n}{A} + \sum_{i} e_i \rad{n}{A} e_i
        \\
        &=
        \aCyc(A) + \Cyc{{} \ge n}{A}.
    \end{align*}
	Then
    \( \kr{n}A = \aCyc(A) + \Cyc{{} \ge n}{A} \)
    by \eqref{eq: trivial inclusion}.
    Thus, by Lemma~\ref{lem: rewrite ubd}, we have
    \begin{align*}
        \codim \kr{n}A
        &= \codim \left( \aCyc(A) + \Cyc{{} \ge n}{A} \right)
		\\
        &= \sum_i \dim e_iAe_i/e_i\rad{n}Ae_i.
        \qedhere
	\end{align*}
\end{proof}

\begin{thm}[See Otokita~\cite{Oto16} for blocks]
    \label{thm: Otokita}
    Let
    \( \{\, e_i \mid 1 \le i \le \ell(A) \,\} \)
    be a basic set of primitive idempotents of \( A \).
    Then the following holds for every \( n \in \nat \).
    \begin{equation*}
        \codim \kr{n}A \le \sum_\id \dim e_i A e_i / e_i \rad{n}{A} e_i
    \end{equation*}

    Furthermore, if the equality holds for some \( n \) then so does for every \( 1 \le m \le n \)\textup{;}
    if the equality does not hold for some \( n \) then so does not for every \( m \ge n \).
\end{thm}

\begin{proof}
    By Theorem~\ref{thm: Morita invariance}, we can assume \( A \) is basic.
    We claim the map
    \begin{align} \begin{split}
        \pi \colon
            &\bigoplus_\id     e_i A e_i / e_i \rad{n}{A} e_i      \to                  A/\kr{n}A,
            \\
            &\quad \sum_i       \big(a_i + e_i \rad{n}{A} e_i\big) \mapsto \sum_i   a_i + \kr{n}A
    \end{split} \end{align}
    is surjective.
    Let \( a + \kr{n}A \in A/\kr{n}A \).
    Then we have the following.
    \begin{align*}
        a + \kr{n}A
            &= \sum_i e_i a e_i + \kr{n}A
                + \sum_{i \neq j} e_i a e_j + \kr{n}A
            \\
            &= \sum_i e_i a e_i + \kr{n}A
                + \sum_{i \neq j} (e_i a e_j - e_j e_i a) + \kr{n}A
            \\
            &= \sum_i e_i a e_i + \kr{n}A
            \\
            &= \pi\bigg(\sum_i e_i a e_i + e_i \rad{n}{A} e_i\bigg).
    \end{align*}
    Hence the proof of the inequality completes.

    The latter statements follow from Proposition~\ref{prop: equal iff}.
\end{proof}

\begin{proof}[Proof of Theorem~\textup{\ref{thm: Brauer Okuyama}}]
    Since the last statement \ref{item: bound for k} is clear from  \ref{item: Okuyama} and Theorem~\ref{thm: Otokita},
    we prove \ref{item: Brauer} and \ref{item: Okuyama} in the following.

    We first claim that equality holds in Theorem~\ref{thm: Otokita} for \( n \le 2 \).
    By Theorem~\ref{thm: Morita invariance},
    we may assume that \( A \) is basic.
    Then we have
     \begin{align*}
        \komm{A}
        &=
        \sum_{i, j, s, t} [e_i A e_j, e_s A e_t]
        & &
        \\
        &\subseteq
        \aCyc(A) + \sum_{i, j} [e_i A e_j, e_j A e_i]
        & &
        \\
        &=
        \aCyc(A) + \sum_{i \neq j} [e_i A e_j, e_j A e_i] + \sum_i \komm{e_i A e_i}
        & &
        \\
        &\subseteq % basic
        \aCyc(A) + \Cyc{{} \ge 2}{A} + \sum_i \komm{e_i A e_i}
        \\
        &\subseteq
        \aCyc(A) + \Cyc{{} \ge 2}{A} + \sum_i \rad{2}{e_i A e_i}
        & & % By K in Rad^2
        \\
        &\subseteq
        \aCyc(A) + \Cyc{{} \ge 2}{A}.
        & &
    \end{align*}
    Now the claim follows from Proposition~\ref{prop: equal iff}.
    Then we have the following.
    \begin{align*}
        \codim \kr{1}{A}
            &= \sum_i \dim e_i A e_i / e_i \radikal(A) e_i
            \\
            &= \ell(A);
            \\
        \codim \kr{2}{A}
            &= \sum_i \dim e_i A e_i / e_i \rad{2}{A} e_i
            \\
            &= \sum_i \dim e_i A e_i / e_i \radikal(A) e_i
            \\
            &\qquad\quad
              + \sum_i \dim e_i \radikal(A) e_i / e_i \rad{2}{A} e_i
            \\
            &= \ell(A) + \sum_i \dim \Ext1(S_i, S_i).
            \qedhere
    \end{align*}
\end{proof}

\begin{note}
	\label{note: about theorem}
    \ref{item: Brauer} If \( F \) is an algebraically closed field of positive characteristic \( p \),
    a well-known theorem of Brauer~\cite[(3A)]{Bra56} states
    \begin{equation*}
        \ell(A) = \codim T(A)
    \end{equation*}
    where
    \begin{equation*}
        T(A) \coloneqq \{\, a \in A \mid \text{\( a^{p^n} \in \komm{A} \) for some \( n \in \nat \)} \,\}.
    \end{equation*}
    Since \( T(A) = \kr{1}A \) by K\"ulshammer~\cite[Lemma~B]{Kue81},
    Theorem~\ref{thm: Brauer Okuyama}\ref{item: Brauer} follows in this case.
    In fact, the part \ref{item: Brauer} is certainly not our contribution although is is sometimes stated only for positive characteristic case;
    It can be found, among others, in \cite[Proposition~I.13.3(i)]{Lan83}.

    \ref{item: Okuyama} Okuyama~\cite{Oku81} proves essentially the same statement of Theorem~\ref{thm: Brauer Okuyama}\ref{item: Okuyama} for a block of a finite group algebra over an algebraically closed field of positive characteristic.
    (See also Koshitani~\cite{Kos16}, which is written in English.)

    \ref{item: bound for k} The inequality \( \ell(A) + \sum_\id \dim \Ext1(S_i, S_i) \le k(A) \) is an extension of Brandt~\cite[Theorem~B]{Bra82}.
    (For a symmetric algebra, it can be proved that the equality holds if and only if its Loewy length is at most two.
    See also \cite[Lemma 2.1]{Sak17}.)
	For a block \( B \) of a finite group algebra, the inequality \( k(B) \le \tr \cartan{B} \) is a direct consequence of
	\( \cartan{B} = {}^tD_B \cdot D_B \)
	where \( D_B \) is the decomposition matrix of \( B \).
	See also K\"ulshammer-Wada~\cite{KW02} for refinements in this case.
\end{note}

\begin{cor}
    \label{cor: rad. sq. zero}
    For a radical square zero algebra \( A \) over a splitting field, we have \( k(A) = \tr \cartan{A} \).
\end{cor}

\begin{proof}
    It suffices to prove \( k(A) \ge \tr \cartan{A} \) by Theorem~\ref{thm: Brauer Okuyama}\ref{item: bound for k}.
    \begin{align*}
        k(A)
        &\ge
        \ell(A) + \sum_i \dim \Ext1(S_i, S_i)
        \qquad \text{(By Theorem~\ref{thm: Brauer Okuyama}\ref{item: bound for k})}
        \\
        &=
        \ell(A) + \sum_i \dim e_i\radikal(A)e_i
        \qquad \text{(By \( \rad{2}A = 0 \))}
        \\
        &=
        \sum_i (1 + \dim e_i\radikal(A)e_i)
		\\
        &=
        \sum_i \dim e_iAe_i
        = \tr \cartan{A}.
        \qedhere
    \end{align*}
\end{proof}

\section{Small algebras}
\label{sec: small algebras}

This section provides a characterization of truncated polynomial algebras \( F[X]/(X^n) \) using Theorems~\ref{thm: Brauer Okuyama} and \ref{thm: Morita invariance}.
Then characterizations of small algebras (Theorems~\ref{thm: k = 1} and \ref{thm: k = 2 and l = 1}) follow immediately.

\begin{lem}
    \label{lem: basic local Nakayama}
    Let \( n \in \nat \).
    Then the following statements are equivalent.
    \begin{enumerate}
        \item \( A \) is isomorphic to \( F[X]/(X^n) \).
        \item \( A \) is \( n \)-dimensional basic local Nakayama algebra.
    \end{enumerate}
\end{lem}

\begin{thm}
    \label{thm: truncated poly. ring}
    Let \( n \in \nat \) and suppose that \( F \) is a splitting field for \( A \).
    Then the following statements are equivalent.
    \begin{enumerate}
        \item \( A \) is Morita equivalent to \( F[X]/(X^n) \).
            \label{item: Morita equiv. to truncated polynomial}
        \item \( k(A) = n \), \( \codim \kr{2}A \le 2 \), and \( \ell(A) = 1 \).
            \label{item: numerical restrictions}
    \end{enumerate}
\end{thm}

\begin{proof}
    \ref{item: Morita equiv. to truncated polynomial} \( \implies \) \ref{item: numerical restrictions}:
    Clear by Theorem~\ref{thm: Morita invariance}.

    \ref{item: numerical restrictions} \( \implies \) \ref{item: Morita equiv. to truncated polynomial}:
    We may assume \( A \) is basic.
    Since \( \ell(A) = 1 \),
    we have the unique simple right \( A \)-module \( S \coloneqq A/\radikal(A) \).
    By Theorem~\ref{thm: Brauer Okuyama},
    \begin{align*}
        \dim \radikal(A)/\rad{2}A
        &=
        \dim \Ext1(S, S)
        \\
        &=
        \ell(A) + \dim \Ext1(S, S) - \ell(A)
        \\
        &=
        \codim \kr{2}A - \codim \kr{1}A
        \\
        &\le
        2 - 1
        =
        1.
    \end{align*}
    Thus \( A \) is a basic local Nakayama algebra.
    By Lemma~\ref{lem: basic local Nakayama},
    we have \( A \cong F[X]/(X^m) \) for some \( m \in \nat \).
    Since \( A \) is commutative,
    \begin{equation*}
        n = \codim \komm{A} = \dim A = \dim F[X]/(X^m) = m.
    \end{equation*}
    Therefore \( A \cong F[X]/(X^n) \).
\end{proof}

\begin{proof}[Proof of Theorem~\textup{\ref{thm: k = 1}}]
    \ref{item: Morita equiv. to F} \( \implies \) \ref{item: k = 1}:
    Clear from Theorem~\ref{thm: Morita invariance}.

    \ref{item: k = 1} \( \implies \) \ref{item: Morita equiv. to F}:
    Clear from Theorem~\ref{thm: truncated poly. ring}.
\end{proof}

\begin{proof}[Proof of Theorem~\textup{\ref{thm: k = 2 and l = 1}}]
    \ref{item: Morita equiv. to F[X]/(X^2)} \( \implies \) \ref{item: k = 2 and l = 1}:
    Clear from Theorem~\ref{thm: Morita invariance}.

    \ref{item: k = 2 and l = 1} \( \implies \) \ref{item: Morita equiv. to F[X]/(X^2)}:
    Clear from Theorem~\ref{thm: truncated poly. ring}.
\end{proof}

\begin{note}
    \label{note: k = 3, l = 1}
    Consider the four-dimensional basic local Frobenius algebra defined by
    \begin{equation*}
        A_q \coloneqq F\langle X, Y \rangle/(X^2, Y^2, XY - qYX)
        \qquad
        (q \in F \setminus \{ 0, 1 \}).
    \end{equation*}
    Since \( A_q \) is non-commutative, we have \( k(A_q) = 3 \) by Theorems~\ref{thm: k = 1} and \ref{thm: k = 2 and l = 1}.
    Hence, we may have infinitely many non-equivalent finite-dimensional algebras \( A \) with \( k(A) = 3 \) and \( \ell(A) = 1 \) unlike the case \( k(A) = 1, 2 \).
	(This is because \( A_q \cong A_r \) if and only if \( \{ q^{\pm1} \} = \{ r^{\pm1} \} \) for \( q, r \in F \setminus \{ 0 \} \).
    See \cite[p.~865]{Yam96}.)
\end{note}

\begin{note}
    \label{note: k = l = 2}
    Consider the \( n \)-Kronecker algebra ---
    the path algebra \( FQ_n \) of the \( n \)-Kronecker quiver \( Q_n \) defined by
    \begin{equation*}
        \begin{tikzcd}
            \circ
            \arrow[r, draw=none, "\raisebox{+1.5ex}{\vdots}" description]
            \arrow[r, bend left,        "\alpha_1"]
            \arrow[r, bend right, swap, "\alpha_n"]
            &
            \circ
        \end{tikzcd}.
    \end{equation*}
    Then \( k(FQ_n) = \ell(FQ_n) = 2 \).
    Hence, we have infinitely many non-equivalent finite-dimensional algebras \( A \) with \( k(A) = \ell(A) = 2 \).
\end{note}

We have determined small algebras in Theorems~\ref{thm: k = 1} and \ref{thm: k = 2 and l = 1}.
The other extreme case is that the algebra \( A \) with \( k(A) = \ell(A) \).
These algebras are close to be semisimple to some extent.

\begin{prop}
    \label{prop: k = l}
    Suppose that \( F \) is a splitting field for \( A \).
    Then the following statements are equivalent.
    \begin{enumerate}
        \item \( k(A) = \ell(A) \)
            \label{item: k = l}
        \item \( \radikal(A) \subseteq \komm{A} \)
            \label{item: J in K}
    \end{enumerate}
    In particular, \( k(A) = \ell(A) \) holds if \( A \) is semisimple.
    The converse also holds if \( A \) is symmetric, commutative, or local.
\end{prop}

\begin{proof}
    \ref{item: k = l} \( \iff \) \ref{item: J in K}:
    Evidently \( k(A) = \ell(A) \) if and only if \( \komm{A} = \kr{1}A \), i.e., \( \radikal(A) \subseteq \komm{A} \).

    Symmetric case:
    Assume \( A \) is symmetric with a symmetrizing form \( \tau \colon A \to F \) and \( k(A) = \ell(A) \).
    Then \( \radikal(A) \subseteq \komm{A} \subseteq \ker \tau \).
    Hence we have \( \radikal(A) = 0 \).

    Commutative case:
    Assume \( A \) is commutative and \( k(A) = \ell(A) \).
    Then \( \radikal(A) \subseteq \komm{A} = 0 \).

    Local case:
    Assume \( A \) is local and \( k(A) = \ell(A) \).
    Then \( k(A) = \ell(A) = 1 \) and \( A \) is Morita equivalent to \( F \) by Theorem~\ref{thm: k = 1}.
    Hence we have \( \radikal(A) = 0 \).
\end{proof}

\appendix

\section{The {C}hlebowitz theorem}
Chlebowitz already studied similar problems and she proved the following stronger facts, 
which are essentially extensions of the previous works~\cite{Kue84, CK92} to non-symmetric cases.
\begin{thm}[Chlebowitz~\cite{Chl91}]
	\label{thm: Chlebowitz}
	Let \( A \) be a finite-dimensional local algebra over an algebraically closed field.
	Then the following holds.
	\begin{enumerate}
		\item If \( k(A) \le 3 \) then \( \dim A \le 4 \).
		\item If \( k(A) = 4 \) then \( \dim A \le 10 \).
		\item If \( k(A) = 5 \) and \( \dim \radikal(A)/\rad{2}{A} \le 2 \) then \( \dim A \le 12 \).
	\end{enumerate}
\end{thm}

\section*{Acknowledgements}
We wish to thank Yoshihiro Otokita for helpful discussions.
This is part of the PhD thesis of the second author.

After the first version of this paper appeared on the arXiv,
Burkhard K\"ulshammer informed us that 
Theorems~\ref{thm: k = 1} and \ref{thm: k = 2 and l = 1} in this version 
already appear in the unpublished PhD thesis by Marlene Chlebowitz~\cite{Chl91}. 
We thank Burkhard K\"ulshammer very much.

\end{document}